\documentclass{article}
\usepackage[utf8]{inputenc}
\usepackage[margin=3cm]{geometry}
\usepackage{hyperref}
\usepackage{amsthm,amsfonts,amssymb,amsmath}
\usepackage{mathtools}  
\usepackage{xcolor,todonotes}

\newtheorem{theorem}{Theorem}[section]
\newtheorem{proposition}[theorem]{Proposition}
\newtheorem{lemma}[theorem]{Lemma}

\newtheorem{claim}[theorem]{Claim}
\newtheorem{corollary}[theorem]{Corollary}
\theoremstyle{definition}
\newtheorem{problem}[theorem]{Problem}

\numberwithin{equation}{section}

\allowdisplaybreaks

\newcommand{\su}{\subseteq}
\newcommand{\sm}{\setminus}

\newcommand{\FF}{\mathbb{F}}

\newcommand{\RR}{\mathbb{R}}

\makeatletter
\newcommand{\pushright}[1]{\ifmeasuring@#1\else\omit\hfill$\displaystyle#1$\fi\ignorespaces}
\makeatother

\begin{document}
\title{Polynomials that vanish to high order on most of the hypercube}
\author{Lisa Sauermann\thanks{Department of Mathematics, Massachusetts Institute of Technology, Cambridge, MA. Email: \url{lsauerma@mit.edu}. This research was supported by the National Science Foundation under Grant CCF-1900460 and under Award DMS-1953772.}\and Yuval Wigderson\thanks{Department of Mathematics, Stanford University, Stanford, CA. Email: \url{yuvalwig@stanford.edu}. Research supported by NSF GRFP Grant DGE-1656518.}}

\maketitle

\begin{abstract}\noindent
Motivated by higher vanishing multiplicity generalizations of Alon's Combinatorial Nullstellensatz and its applications, we study the following problem: for fixed $k\geq 1$ and $n$ large with respect to $k$, what is the minimum possible degree of a polynomial $P\in \RR[x_1,\dots,x_n]$ with $P(0,\dots,0)\neq 0$ such that $P$ has zeroes of multiplicity at least $k$ at all points in $\{0,1\}^n\sm \{(0,\dots,0)\}$? For $k=1$, a classical theorem of Alon and Füredi states that the minimum possible degree of such a polynomial equals $n$. In this paper, we solve the problem for all $k\geq 2$, proving that the answer is $n+2k-3$. As an application, we improve a result of Clifton and Huang on configurations of hyperplanes in $\RR^n$ such that each point in $\{0,1\}^n\sm \{(0,\dots,0)\}$ is covered by at least $k$ hyperplanes, but the point $(0,\dots,0)$ is uncovered. Surprisingly, the proof of our result involves Catalan numbers and arguments from enumerative combinatorics.
\end{abstract}

\section{Introduction}

Alon's Combinatorial Nullstellensatz \cite{alon}, which gives a non-vanishing criterion for a polynomial on some grid of points under certain conditions, has had applications to many problems in combinatorics. The following statement due to Alon and Füredi \cite{alon-furedi} is a now classical example of an application of the Combinatorial Nullstellensatz (even though it historically predated it).

\begin{theorem}[see \cite{alon-furedi}]\label{thm-alon-furedi}
Let $n\geq 1$. Then any polynomial $P\in \RR[x_1,\dots,x_n]$ with $P(0,\dots,0)\neq 0$ and such that $P$ has zeroes at all points in $\{0,1\}^n\sm \{(0,\dots,0)\}$ has degree $\deg P\geq n$.
\end{theorem}

The example $P=(x_1-1)\dotsm (x_n-1)$ shows that the bound $\deg P\geq n$ is tight. Alon and Füredi \cite{alon-furedi} proved this theorem in order to solve a problem of Komj\'{a}th \cite{komjath} asking about the minimum possible number $m$ such that there are $m$ hyperplanes in $\RR^n$ covering all points in $\{0,1\}^n\sm \{(0,\dots,0)\}$, but not covering $(0,\dots,0)$. By considering the product of the linear polynomials defining these hyperplanes, Theorem \ref{thm-alon-furedi} easily implies that at least $n$ hyperplanes are needed. As $n$ hyperplanes are also sufficient, the answer to Komj\'{a}th's hyperplane problem is $m=n$.

There has been a lot of work on finding generalizations of both the Combinatorial Nullstellensatz \cite{alon} and of Alon and Füredi's result \cite{alon-furedi} on Komj\'{a}th's hyperplane problem to higher vanishing or covering multiplicities \cite{ball-serra, batzaya-bayarmagnnai,clifton-huang,kos-ronyai, kos-meszaros-ronyai}. 
More generally, there is now a rich collection of  higher-multiplicity generalizations of related algebraic results. There are also various combinatorial applications of such higher-multiplicity results, like Stepanov's method (see e.g. \cite{hanson-petridis}) and the multiplicity Schwartz--Zippel lemma (see e.g.\ \cite{bishnoi-etal,bukh-chao,dvir-etal}).
In this spirit, it is natural to also ask for generalizations of Theorem \ref{thm-alon-furedi} to higher vanishing orders for the polynomial $P$. This leads to the following problem.

\begin{problem}\label{problem-main}
Let $k\geq 2$ and let $n$ be sufficiently large with respect to $k$. What is the minimum degree that a polynomial $P\in \RR[x_1,\dots,x_n]$ can have, if $P(0,\dots,0)\neq 0$ and if $P$ has zeroes of multiplicity at least $k$ at all points in $\{0,1\}^n\sm \{(0,\dots,0)\}$?
\end{problem}

As usual, we say that a polynomial $P\in \RR[x_1,\dots,x_n]$ has a zero of multiplicity at least $k$ at a point $a\in \RR^n$ if all derivatives of $P$ up to order $k-1$ vanish at $a$. Note that $P(a)=0$ is equivalent to $P$ having a zero of multiplicity at least $1$ at $a$.

Ball and Serra \cite[Theorem 4.1]{ball-serra} proved a lower bound of $\deg P\geq n+k-1$ for Problem \ref{problem-main} (in fact, they proved a similar bound for more general grids instead of just $\{0,1\}^n\su \RR^n$). This in particular implies that the answer to Problem \ref{problem-main} is $n+1$ if $k=2$. Clifton and Huang \cite{clifton-huang} proved that the answer to Problem \ref{problem-main} is $n+3$ if $k=3$, and they improved the lower bound to $\deg P\geq n+k+1$ for $k\geq 4$. Clifton and Huang were actually studying the generalization of Komj\'{a}th's hyperplane problem \cite{komjath} mentioned above to higher covering multiplicities (where every point in $\{0,1\}^n\sm \{(0,\dots,0)\}$ needs to be covered by at least $k$ hyperplanes, while $(0,\ldots,0)$ must remain uncovered). Their approach was to consider the product of the polynomials for all of the hyperplanes and to prove a lower bound for the degree of this product. This naturally leads to Problem \ref{problem-main}, even though their lower bound results are not explicitly stated in the setting of Problem \ref{problem-main}.

In this paper, we resolve Problem \ref{problem-main}, showing that the answer is $n+2k-3$. This is the content of the following theorem.

\begin{theorem}\label{thm-1}
Let $k\geq 2$ and $n\geq 2k-3$. Then any polynomial $P\in \RR[x_1,\dots,x_n]$ with $P(0,\dots,0)\neq 0$ having zeroes of multiplicity at least $k$ at all points in $\{0,1\}^n\sm \{(0,\dots,0)\}$ has degree $\deg P\geq n+2k-3$. Furthermore, there exists such a polynomial $P$ with degree $\deg P= n+2k-3$.
\end{theorem}

Instead of demanding $P(0,\dots,0)\neq 0$, one can also ask a more general version of Problem \ref{problem-main} where $P$ is required to have a zero of multiplicity exactly $\ell$ at $(0,\dots,0)$ for some $\ell\in \{0,\dots,k-1\}$ (note that for $\ell\geq k$, we cannot expect any interesting lower bounds on the degree of $P$ for large $n$, as the example $P=x_1^\ell(x_1-1)^k$ with $\deg P=\ell+k$ shows). Also note that the case $\ell=0$ corresponds to Problem \ref{problem-main} above. Our next theorem resolves this more general problem for $0\leq \ell \leq k-2$.

\begin{theorem}\label{thm-2}
Let $k\geq 2$ and $n\geq 2k-3$. Let $P\in \RR[x_1,\dots,x_n]$ be a polynomial having zeroes of multiplicity at least $k$ at all points in $\{0,1\}^n\sm \{(0,\dots,0)\}$, and such that $P$ does not have a zero of multiplicity at least $k-1$ at $(0,\dots,0)$. Then $P$ must have degree $\deg P\geq n+2k-3$. Furthermore, for every $\ell=0,\dots,k-2$, there exists a polynomial $P$ with degree $\deg P= n+2k-3$ having zeroes of multiplicity at least $k$ at all points in $\{0,1\}^n\sm \{(0,\dots,0)\}$, and such that $P$ has a zero of multiplicity exactly $\ell$ at $(0,\dots,0)$.
\end{theorem}

Note that Theorem \ref{thm-2} implies Theorem \ref{thm-1}. Indeed, it is clear that the first part of Theorem \ref{thm-2} (giving the degree bound $\deg P\geq n+2k-3$) implies the first part of Theorem \ref{thm-1}. Furthermore, the second part of Theorem \ref{thm-1} is equivalent to the second part of Theorem \ref{thm-2} for $\ell=0$.

Theorem \ref{thm-2} does not address the case where $P$ has a zero of multiplicity exactly $k-1$ at $(0,\dots,0)$. This case turns out to have a slightly different answer, as shown in the following theorem.

\begin{theorem}\label{thm-3}
Let $k\geq 2$ and $n\geq 1$. Let $P\in \RR[x_1,\dots,x_n]$ be a polynomial having zeroes of multiplicity at least $k$ at all points in $\{0,1\}^n\sm \{(0,\dots,0)\}$, and a zero of multiplicity exactly $k-1$ at $(0,\dots,0)$. Then $P$ must have degree $\deg P\geq n+2k-2$. Furthermore, there exists such a polynomial $P$ with degree $\deg P= n+2k-2$.
\end{theorem}

Note that Theorem \ref{thm-3} is also true for $k=1$, and is identical to Alon and F\"uredi's result in Theorem \ref{thm-alon-furedi}. Theorem \ref{thm-3} is actually significantly easier to prove than Theorems \ref{thm-1} and \ref{thm-2}.

As an application of Theorem \ref{thm-1}, we can improve a result of Clifton and Huang \cite{clifton-huang} concerning collections of hyperplanes in $\RR^n$ such that every point in $\{0,1\}^n\sm \{(0,\dots,0)\}$ is covered by at least $k$ of these hyperplanes, but no hyperplane contains $(0,\dots,0)$. Generalizing Komj\'ath's hyperplane problem \cite{komjath}, Clifton and Huang studied the minimum possible size of such a collection of hyperplanes, where $k$ is fixed and $n$ is assumed to be sufficiently large with respect to $k$. They proved that the minimum size of such a collection of hyperplanes is $n+1$ if $k=2$ and $n+3$ if $k=3$. Furthermore, for $k\geq 4$ they proved a lower bound of $n+k+1$ for the size of any such collection of hyperplanes. As an immediate corollary of Theorem \ref{thm-1}, we can recover their lower bound for $k\in \{2,3,4\}$ and improve it for $k\geq 5$.

\begin{corollary}\label{coro-hyperplanes}
Fix $k\geq 2$ and let $n\geq 2k-3$. Consider a collection of hyperplanes in $\RR^n$ such that every point in $\{0,1\}^n\sm \{(0,\dots,0)\}$ is covered by at least $k$ of these hyperplanes, but no hyperplane contains $(0,\dots,0)$. Then this collection must consist of at least $n+2k-3$ hyperplanes.
\end{corollary}

Indeed, Corollary \ref{coro-hyperplanes} follows from Theorem \ref{thm-1} by taking the linear polynomials corresponding to all of the hyperplanes, and observing that their product $P$ satisfies the assumptions in Theorem \ref{thm-1}. This was also the approach taken by Clifton and Huang \cite{clifton-huang} to prove their lower bounds for the size of any such collection of hyperplanes. Clifton and Huang obtained lower bounds on the degree of the resulting product polynomial $P$ by using a punctured higher-multiplicity version of the Combinatorial Nullstellensatz due to Ball and Serra \cite{ball-serra} and a relatively involved analysis of the conditions on the coefficients of the polynomial $P$ imposed by its vanishing properties. As mentioned above, their arguments also apply in the setting of Problem \ref{problem-main}, but their bounds on the degree of $P$ are weaker than the lower bound we obtain in Theorem \ref{thm-1}.

It is still an interesting open problem to determine the minimum possible size of a collection of hyperplanes as in Corollary \ref{coro-hyperplanes} for $k\geq 4$ (and $n$ sufficiently large with respect to $k$). While Corollary \ref{coro-hyperplanes} gives the best currently known lower bound, the best known upper bound is $n+\binom{k}{2}$, obtained by a construction due to Clifton and Huang \cite{clifton-huang}. They conjectured that this upper bound is tight if $n$ is sufficiently large with respect to $k$; see also the discussion in Subsection \ref{subsec:clifton-huang} in the concluding remarks.

While Theorem \ref{thm-3} is relatively easy to prove, the proof of Theorem \ref{thm-2} takes up most of this paper. Indeed, the proof of Theorem \ref{thm-3} uses only fairly standard techniques related to the Combinatorial Nullstellensatz and its generalizations (like iteratively subtracting appropriate monomials in order to put a given polynomial in some ``canonical form''). Proving Theorem \ref{thm-2} is, however, significantly harder. With various linear algebra arguments, one can reduce the desired statement to showing that a certain linear map is an isomorphism (see Corollary \ref{coro-phi-injective} below). In order to show surjectivity of this linear map, we analyze the representations of certain symmetric polynomials in terms of power sum symmetric polynomials. We then need to show that a certain coefficient of such a representation is non-zero. Surprisingly, it turns out that this coefficient is actually equal (up to sign) to a Catalan number. The fact that enumerative combinatorics arguments appear in our proof is maybe somewhat unexpected, given that Problem \ref{problem-main} is a problem in extremal combinatorics. 

\textit{Notation.} Throughout this paper, we work with the usual convention that the binomial coefficients $\binom{n}{m}$ are defined for all integers $m$ and $n\geq 0$, but we have $\binom{n}{m}=0$ unless $0\leq m\leq n$. The variables $i,j,k,\ell,m,n,r,s,t$ always refer to integers.

\section{Proof of Theorems \ref{thm-2} and \ref{thm-3}}
\label{sect-proof}

In this section we prove Theorem \ref{thm-3}, and we also prove Theorem \ref{thm-2} apart from the proof of Proposition \ref{propo-key} below, which we postpone to Section \ref{sect-proof-proposition}. Recall that Theorem \ref{thm-2} implies Theorem \ref{thm-1}.

For $k\geq 2$, let us say that a polynomial $P\in \RR[x_1,\dots,x_n]$ is \emph{$k$-reduced}, if $\deg P\leq n+2k-3$ and if $P$ does not contain any monomials divisible by $x_{i_1}^2\dotsm x_{i_k}^2$ for some (not necessarily distinct) indices $i_1,\dots,i_k\in \{1,\dots,n\}$ (in other words, no monomial of $P$ is divisible by a product of $k$ squares of variables). Let $U_k\su \RR[x_1,\dots,x_n]$ be the vector space of all $k$-reduced polynomials.

Before starting the proofs of Theorems \ref{thm-2} and \ref{thm-3}, we will show a sequence of claims. The first claim shows that in order to prove the first part of Theorems \ref{thm-2} and \ref{thm-3}, we can restrict ourselves to considering $k$-reduced polynomials $P$. Furthermore, it will be easy to derive the second part of Theorem \ref{thm-2} from this claim.

\begin{claim}\label{claim-reduced-exists}
Let $k\geq 2$ and $n\geq 1$. For every polynomial $Q\in \RR[x_1,\dots,x_n]$, we can find a polynomial $P\in U_k$ with $\deg P\leq \deg Q$ such that the difference $Q-P$ has zeroes of multiplicity at least $k$ at all points in $\{0,1\}^n\sm \{(0,\dots,0)\}$ and a zero of multiplicity at least $k-1$ at $(0,\dots,0)$. Furthermore, if $\deg Q\leq n+2k-3$, then we can even choose $P\in U_k$ with $\deg P\leq \deg Q$ such that the difference $Q-P$ has zeroes of multiplicity at least $k$ at all points in $\{0,1\}^n$.
\end{claim}

\begin{proof}
Fix $k\geq 2$ and $n\geq 1$. Note that for any polynomial $Q\in \RR[x_1,\dots,x_n]$, when defining a polynomial $Q^*\in \RR[x_1,\dots,x_n]$ by either
\begin{itemize}
    \item[(a)] $Q^*=Q-a\cdot x_{i_1}(x_{i_1}-1)\dotsm x_{i_k}(x_{i_k}-1)\cdot x_1^{m_1}\dotsm x_n^{m_n}$ with $a\in \RR$, with non-negative integers $m_1,\dots,m_n$ and with (not necessarily distinct) indices $i_1,\dots,i_k\in \{1,\dots,n\}$, or
    \item[(b)] $Q^*=Q-a\cdot x_{i_1}(x_{i_1}-1)\dotsm x_{i_{k-1}}(x_{i_{k-1}}-1)\cdot (x_1-1)\dotsm (x_n-1)$ with $a\in \RR$ and with (not necessarily distinct) indices $i_1,\dots,i_{k-1}\in \{1,\dots,n\}$,
\end{itemize}
the difference $Q-Q^*$ has zeroes of multiplicity at least $k$ at all points in $\{0,1\}^n\sm \{(0,\dots,0)\}$ and a zero of multiplicity at least $k-1$ at $(0,\dots,0)$. Furthermore, if we define $Q^*$ as in (a), then $Q-Q^*$ even has zeroes of multiplicity at least $k$ at all points in $\{0,1\}^n$.

Let us say that a monomial in $\RR[x_1,\dots,x_n]$ is \emph{bad} if it is of the form $x_{i_1}^2\dotsm x_{i_k}^2\cdot x_1^{m_1}\dotsm x_n^{m_n}$ with non-negative integers $m_1,\dots,m_n$ and with (not necessarily distinct) indices $i_1,\dots,i_k\in \{1,\dots,n\}$, or of the form $x_{i_1}^2\dotsm x_{i_{k-1}}^2\cdot x_1\dotsm x_n$  with (not necessarily distinct) indices $i_1,\dots,i_{k-1}\in \{1,\dots,n\}$.

Whenever $Q\in \RR[x_1,\dots,x_n]$ contains a bad monomial, we can apply one of the steps (a) or (b) above, replacing the bad monomial with monomials of lower degree. We can repeatedly perform these steps, always replacing a bad monomial of maximum degree by lower-degree monomials, until we arrive at a polynomial $P\in \RR[x_1,\dots,x_n]$ which does not contain any bad monomials. Indeed, note that when repeatedly applying these steps, the process must terminate at some point (since at every step either the maximum degree of the occurring bad monomials decreases or the maximum degree of the occurring bad monomials remains unchanged but the number of different bad monomials of this maximum degree decreases).

By construction of $P$, we have $\deg P\leq \deg Q$ and the difference $Q-P$ has zeroes of multiplicity at least $k$ at all points in $\{0,1\}^n\sm \{(0,\dots,0)\}$ and a zero of multiplicity at least $k-1$ at $(0,\dots,0)$. Let us now check that $P\in U_k$, i.e.\ that $P$ is $k$-reduced and has degree $\deg P\leq n+2k-3$.

Recall that $P$ does not contain any bad monomials. This means in particular that $P$ does not have any monomials which are divisible by $x_{i_1}^2\dotsm x_{i_k}^2$ for some (not necessarily distinct) indices $i_1,\dots,i_k\in \{1,\dots,n\}$. Note that this already implies that $\deg P\leq n+2k-2$. Furthermore, if $P$ had a monomial of degree $n+2k-2$, then this monomial would need to be of the form $x_{i_1}^2\dotsm x_{i_{k-1}}^2\cdot x_1\dotsm x_n$ with (not necessarily distinct) indices $i_1,\dots,i_{k-1}\in \{1,\dots,n\}$. But such a monomial is also bad, and therefore we must have $\deg P\leq n+2k-3$. This proves that $P\in U_k$.

Finally, for the second part of the claim, note that if we have $\deg Q\leq n+2k-3$, then throughout the process of obtaining $P$ from $Q$ all polynomials have degree at most $n+2k-3$ and we therefore never apply step (b). This means that we will only perform step (a), and hence $Q-P$ has zeroes of multiplicity at least $k$ at all points in $\{0,1\}^n$.
\end{proof}

For $n\geq 1$ and $t\geq 1$, let us define $M_t(n)$ to be the number of $n$-tuples $(m_1,\dots,m_n)$ of non-negative integers with $m_1+\dots+m_n<t$. We remark that $M_t(n)= \binom{n+t-1}n$, but we will not use this formula. Note that for a polynomial $P\in \RR[x_1,\dots,x_n]$ there are precisely $M_t(n)$ different ways to form a derivative of $P$ of order less than $t$.

We will be using various dimension-counting arguments. 
The following claim expresses the dimension of the vector space $U_k$ in terms of the numbers $M_t(n)$ we just defined.

\begin{claim}\label{claim-dim-Uk}
Let $k\geq 2$ and $n\geq 1$. Then the vector space $U_k\su \RR[x_1,\dots,x_n]$ has dimension $\dim U_k=(2^n-1)\cdot M_k(n)+M_{k-1}(n)$.
\end{claim}

\begin{proof}
The dimension of the vector space $U_k$ equals the number of monomials in $\RR[x_1,\dots,x_n]$ of degree at most $n+2k-3$ which are not divisible by $x_{i_1}^2\dotsm x_{i_k}^2$ for any (not necessarily distinct) indices $i_1,\dots,i_k\in \{1,\dots,n\}$. In other words, $\dim U_k$ is the number of monomials of the form $x_1^{2m_1+r_1}\dotsm x_n^{2m_n+r_n}$ with non-negative integers  $(m_1,\dots,m_n)$ with $m_1+\dots+m_n<k$ and $r_1,\dots,r_n\in \{0,1\}$ such that $(2m_1+r_1)+\dots +(2m_n+r_n)\leq n+2k-3$. If $r_1+\dots+r_n\leq n-1$, then there are precisely $M_k(n)$ choices for $(m_1,\dots,m_n)$. But if $r_1=\dots=r_n=1$, then we must have $2m_1+\dots +2m_n\leq 2k-3$, which (for non-negative integers $m_1,\dots,m_n$) is equivalent to $m_1+\dots+m_n<k-1$. Hence, in this case, there are only $M_{k-1}(n)$ choices for $(m_1,\dots,m_n)$. All in all we obtain $\dim U_k=(2^n-1)\cdot M_k(n)+M_{k-1}(n)$, as desired.
\end{proof}

\begin{claim}\label{claim-isomorphism-evaluations}
Let $k\geq 2$ and $n\geq 1$, and let $N=(2^n-1)\cdot M_k(n)+M_{k-1}(n)$. Consider the linear map $\psi_k:U_k\to \RR^N$ sending each polynomial $P\in U_k$ to the $N$-tuple consisting of the derivatives of $P$ of order less than $k$ at all points in $\{0,1\}^n\sm \{(0,\dots,0)\}$ and all the derivatives of $P$ of order less than $k-1$ at $(0,\dots,0)$. Then  $\psi_k:U_k\to \RR^N$ is an isomorphism.
\end{claim}
\begin{proof}
By Claim \ref{claim-dim-Uk} we have $\dim U_k=(2^n-1)\cdot M_k(n)+M_{k-1}(n)=N$, so it suffices to prove that the linear map $\psi_k:U_k\to \RR^N$ is surjective. The surjectivity of $\psi_k$ follows from Claim \ref{claim-reduced-exists}: indeed, for any $N$-tuple $\alpha\in \RR^N$, we can easily construct a (high-degree) polynomial $Q\in \RR[x_1,\dots,x_n]$ such that the derivatives of $Q$ of order less than $k$ at all points in $\{0,1\}^n\sm \{(0,\dots,0)\}$ and the derivatives of $Q$ of order less than $k-1$ at $(0,\dots,0)$ form precisely the $N$-tuple $\alpha$. For example, one can take
\[Q=\sum \frac{\alpha_{(a_1,\dots,a_n),(m_1,\dots,m_n)}}{m_1!\dotsm m_n!}\cdot (x_1-a_1)^{m_1}\dotsm (x_n-a_n)^{m_n}\cdot ((x_1-a_1)^{2k}-1)^{2k}\dotsm ((x_n-a_n)^{2k}-1)^{2k},\]
where the sum is over all entries $\alpha_{(a_1,\dots,a_n),(m_1,\dots,m_n)}$ of $\alpha$ (which correspond to conditions of the form $(\partial_{x_1})^{m_1}\dots (\partial_{x_n})^{m_n} Q(a_1,\dots,a_n)=\alpha_{(a_1,\dots,a_n),(m_1,\dots,m_n)}$ with $(a_1,\dots,a_n)\in \{0,1\}^n$).
Now, by Claim \ref{claim-reduced-exists} there exists a polynomial $P\in U_k$ where this $N$-tuple of derivatives agrees with the $N$-tuple for $Q$. Hence $\psi_k(P)=\alpha$.
\end{proof}

\begin{corollary}\label{corollary-reduced-unique}
Let $k\geq 2$ and $n\geq 1$. For every polynomial $Q\in \RR[x_1,\dots,x_n]$ there is at most one polynomial $P\in U_k$ such that $Q-P$ has zeroes of multiplicity at least $k$ at all points in $\{0,1\}^n\sm \{(0,\dots,0)\}$ and a zero of multiplicity at least $k-1$ at $(0,\dots,0)$. In particular, the polynomial $P$ in Claim \ref{claim-reduced-exists} is unique.
\end{corollary}
\begin{proof}
Suppose that for some polynomial $Q\in \RR[x_1,\dots,x_n]$ there were two different polynomials $P, P^*\in U_k$ with the desired properties. Then the non-zero polynomial $P-P^*\in U_k$ would have zeroes of multiplicity at least $k$ at all points in $\{0,1\}^n\sm \{(0,\dots,0)\}$ and a zero of multiplicity at least $k-1$ at $(0,\dots,0)$, and would therefore be mapped to $(0,0,\dots,0)\in \RR^N$ under the isomorphism $\psi_k:U_k\to \RR^N$ in Claim \ref{claim-isomorphism-evaluations}. This is a contradiction.
\end{proof}

With these preparations, we are now ready to prove Theorem \ref{thm-3}.

\begin{proof}[Proof of Theorem \ref{thm-3}]
Suppose for contradiction that there is a polynomial $Q\in \RR[x_1,\dots,x_n]$ of degree $\deg Q\leq n+2k-3$ having zeroes of multiplicity at least $k$ at all points in $\{0,1\}^n\sm \{(0,\dots,0)\}$, and a zero of multiplicity exactly $k-1$ at $(0,\dots,0)$. By the second part of Claim \ref{claim-reduced-exists}, there is a polynomial $P\in U_k$ also having zeroes of multiplicity at least $k$ at all points in $\{0,1\}^n\sm \{(0,\dots,0)\}$, and a zero of multiplicity exactly $k-1$ at $(0,\dots,0)$. In particular, $P$ is not the zero polynomial. However, under the isomorphism $\psi_k: U_k\to \RR^N$ in Claim \ref{claim-isomorphism-evaluations}, the polynomial $P$ is mapped to  $(0,0,\dots,0)\in \RR^N$ (recall that the derivatives of order $k-1$ at $(0,\dots,0)$ are not recorded by $\psi_k$). This is a contradiction.

For the second part of the theorem, note that the polynomial
\[P(x_1,\dots,x_n)=x_1^{k-1}(x_1-1)^{k-1}\cdot (x_1-1)\dotsm(x_n-1)\]
has degree $\deg P=n+2k-2$ and satisfies the desired conditions.
\end{proof}

It remains to prove Theorem \ref{thm-2}. The second part of Theorem \ref{thm-2} (concerning the existence of $P$ with $\deg P=n+2k-3$ with the desired conditions) is a direct consequence of Claim \ref{claim-reduced-exists}, as we will see below. However, the first part of Theorem \ref{thm-2} is much more challenging.

For $k\geq 2$, let $V_k\su U_k\su \RR[x_1,\dots,x_n]$ be the vector space of all $k$-reduced polynomials which have zeroes of multiplicity at least $k$ at all points in $\{0,1\}^n\sm \{(0,\dots,0)\}$. In order to prove the first part of Theorem \ref{thm-2}, our goal will be to show that all non-zero polynomials $P\in V_k$ have degree $n+2k-3$. We continue with our sequence of claims.

\begin{claim}\label{claim-dim-Vk}
Let $k\geq 2$ and $n\geq 1$. Then the vector space $V_k\su \RR[x_1,\dots,x_n]$ has dimension $M_{k-1}(n)$.
\end{claim}
\begin{proof}
Note that $V_k$ consists precisely of those polynomials in $U_k$ which are mapped to $N$-tuples under the isomorphism $\psi_k: U_k\to \RR^N$ in Claim \ref{claim-isomorphism-evaluations} where the first $(2^n-1)\cdot M_k(n)$ entries are zero. Since the subspace of such $N$-tuples has dimension $N-(2^n-1)\cdot M_k(n)=M_{k-1}(n)$, the subspace $V_k\su U_k$ also has dimension $M_{k-1}(n)$.
\end{proof}

\begin{claim}\label{claim-multiplication-increase-k}
Let $k\geq 3$ and $n\geq 1$. Then for each $j\in \{1,\dots,n\}$ and each polynomial $P\in V_{k-1}$, we have $x_j(x_j-1)\cdot P\in V_k$.
\end{claim}
\begin{proof}
Recall that $P\in V_{k-1}\su U_{k-1}$ is $(k-1)$-reduced, meaning that $\deg P\leq n+2k-5$ and no monomial of $P$ is divisible by $x_{i_1}^2\dotsm x_{i_{k-1}}^2$ for any (not necessarily distinct) indices $i_1,\dots,i_{k-1}\in \{1,\dots,n\}$. Then clearly the polynomial $x_j(x_j-1)\cdot P$ has degree at most $n+2k-3$, and it is also not hard to see that no monomial of $x_j(x_j-1)\cdot P$ is divisible by $x_{i_1}^2\dotsm x_{i_{k}}^2$ for any (not necessarily distinct) indices $i_1,\dots,i_{k}\in \{1,\dots,n\}$. Hence $x_j(x_j-1)\cdot P\in U_k$.

Since $P\in V_{k-1}$, the polynomial $P$ has zeroes of multiplicity at least $k-1$ at all points in $\{0,1\}^n\sm \{(0,\dots,0)\}$. Hence $x_j(x_j-1)\cdot P$ has zeroes of multiplicity at least $k$ at all these points, and we can conclude that $x_j(x_j-1)\cdot P\in V_k$.
\end{proof}

Now, in order to show that all non-zero polynomials $P\in V_k$ have degree $n+2k-3$, let us consider the linear map $\varphi_k:V_k\to \RR[x_1,\dots,x_n]$ sending each polynomial in $V_k$ to its homogeneous degree $n+2k-3$ part. Our goal is to show that this map is injective. To show this, it is sufficient to prove that the image $\varphi_k(V_k)$ has dimension $\dim V_k=M_{k-1}(n)$.

Let $W_k\su \RR[x_1,\dots,x_n]$ be the subspace spanned by all polynomials of the form
\begin{equation}\label{eq-generators-Wk}
x_1\dotsm x_n \cdot (x_1^m+\dots+x_n^m)\cdot x_1^{2d_1}\dotsm x_n^{2d_n}
\end{equation}
for non-negative integers $(m,d_1,\dots,d_n)$  with  $m+2(d_1+\dots+d_n)=2k-3$. Note that all polynomials in $W_k$ are homogeneous of degree $n+2k-3$. We will later show that $\varphi_k(V_k)=W_k$.

\begin{claim}\label{claim-basis-Wk}
Let $k\geq 2$ and $n\geq k-1$. Then the polynomials in (\ref{eq-generators-Wk}) form a basis of $W_k$.
\end{claim}
\begin{proof}
By definition of $W_k$, the polynomials in (\ref{eq-generators-Wk}) span $W_k$. It remains to show that these polynomials are linearly independent. Since we can divide all the polynomials by $x_1\dotsm x_n$, it suffices to show that the polynomials of the form
\[Q_{m,d_1,\dots,d_n}=(x_1^m+\dots+x_n^m)\cdot x_1^{2d_1}\dotsm x_n^{2d_n}\]
for non-negative integers $(m,d_1,\dots,d_n)$  with  $m+2(d_1+\dots+d_n)=2k-3$ are linearly independent.

Note that for each such polynomial $Q_{m,d_1,\dots,d_n}$ we have that $m$ is odd and that $d_1+\dots+d_n\leq k-2<n$. Hence all monomials of $Q_{m,d_1,\dots,d_n}$ have exactly one variable with an odd exponent, this odd exponent is always at least $m$, and there exists some monomial where the odd exponent is equal to $m$ (namely each of the monomials $x_i^m\cdot x_1^{2d_1}\dotsm x_n^{2d_n}$ for those $i\in \{1,\dots,n\}$ with $d_i=0$, which must exist since $d_1+\dots+d_n<n$).

Now suppose we had a linear dependence relationship
\[\sum_{(m,d_1,\dots,d_n)}\lambda_{m,d_1,\dots,d_n}Q_{m,d_1,\dots,d_n}=0\]
for some coefficients $\lambda_{m,d_1,\dots,d_n}\in \RR$ not all of which are zero. Then let $m^*$ be the minimum value of $m$ for which some coefficient $\lambda_{m,d_1,\dots,d_n}$ is non-zero (and note that $m^*$ is odd, since $m$ is always odd). Furthermore, fix $i\in \{1,\dots,n\}$ such that we have $\lambda_{m,d_1,\dots,d_n}\neq 0$ for some $(m,d_1,\dots,d_n)$ with $m=m^*$ and $d_i=0$.

Now, in the polynomials $Q_{m,d_1,\dots,d_n}$ with $\lambda_{m,d_1,\dots,d_n}\neq 0$, all monomials have exactly one variable with an odd exponent, and this odd exponent is always at least $m^*$. Let us consider all the monomials where $x_i$ has exponent $m^*$. These monomials can only appear in polynomials $Q_{m,d_1,\dots,d_n}$ with $\lambda_{m,d_1,\dots,d_n}\neq 0$ where $m=m^*$ and $d_i=0$. But note that each such polynomial contains precisely one monomial where $x_i$ has exponent $m=m^*$, namely the monomial $x_i^{m^*}\cdot x_1^{2d_1}\dotsm x_n^{2d_n}$. Hence these monomials cannot cancel out between different polynomials $Q_{m,d_1,\dots,d_n}$ with $\lambda_{m,d_1,\dots,d_n}\neq 0$. Thus, a linear dependence relationship between the polynomials $Q_{m,d_1,\dots,d_n}$ is not possible.
\end{proof}

\begin{corollary}\label{coro-dim-Wk}
Let $k\geq 2$ and $n\geq k-1$. Then the vector space $W_k\su \RR[x_1,\dots,x_n]$ has dimension $M_{k-1}(n)$.
\end{corollary}
\begin{proof}
By Claim \ref{claim-basis-Wk}, the dimension of $W_k$ equals the number of polynomials of the form as in (\ref{eq-generators-Wk}), i.e.\ the number of choices of non-negative integers $(m,d_1,\dots,d_n)$ with $m+2(d_1+\dots+d_n)=2k-3$. We must have $d_1+\dots+d_n<k-1$. So there are $M_{k-1}(n)$ possibilities to choose $(d_1,\dots,d_n)$, and for each of them $m=2k-3-2(d_1+\dots+d_n)$ is uniquely determined.
\end{proof}

A key step in the proof of Theorem \ref{thm-2} is to prove the following proposition, from which we will deduce that the map $\varphi_k: V_k\to \RR[x_1,\dots,x_n]$ is injective.

\begin{proposition}\label{propo-key}
Let $k\geq 2$ and $n\geq 2k-3$. Then there exists a polynomial $P\in V_k$ such that the homogeneous degree $n+2k-3$ part $\varphi_k(P)$ of $P$ satisfies $\varphi_k(P)\in W_k$ as well as the following condition: if we write $\varphi_k(P)\in W_k$ in terms of the basis in (\ref{eq-generators-Wk}), then the coefficient of the basis element $x_1\dotsm x_n \cdot (x_1^{2k-3}+\dots+x_n^{2k-3})$ is non-zero.
\end{proposition}

We postpone the proof of Proposition \ref{propo-key} to Section \ref{sect-proof-proposition}. The proposition implies the following corollaries.

\begin{corollary}\label{coro-Wk-in-image}
Let $k\geq 2$ and $n\geq 2k-3$. Then we have $W_k\su \varphi_k(V_k)$.
\end{corollary}
\begin{proof}
We prove the corollary by induction on $k$. For the base case $k=2$ note that the space $W_2$ is the one-dimensional space spanned by the polynomial $x_1\dotsm x_n \cdot(x_1+\dots+x_n)$. By Proposition \ref{propo-key} for $k=2$, there exists a polynomial $P\in V_2$ such that $\varphi_2(P)\in W_2$ is a non-zero scalar multiple of $x_1\dotsm x_n \cdot(x_1+\dots+x_n)$. This shows that $W_2\su \varphi_2(V_2)$.

Now let $k\geq 3$ and assume that we have already shown that $W_{k-1}\su \varphi_{k-1}(V_{k-1})$. Let $W_k'\su W_k$ be the subspace spanned by all polynomials of the form (\ref{eq-generators-Wk}) with $d_1+\dots+d_n\geq 1$. In other words, $W_k'\su W_k$ is spanned by all of the form (\ref{eq-generators-Wk}) except the polynomial $x_1\dotsm x_n \cdot (x_1^{2k-3}+\dots+x_n^{2k-3})$. Note that $\dim W_k'=\dim W_k -1$.

Let us first show that $W_k'\su \varphi_k(V_k)$. So fix a polynomial $x_1\dotsm x_n \cdot (x_1^m+\dots+x_n^m)\cdot x_1^{2d_1}\dotsm x_n^{2d_n}$ of the form (\ref{eq-generators-Wk}) where $d_1+\dots+d_n\geq 1$. Now, let $j\in \{1,\dots,n\}$ be chosen such that $d_j\geq 1$. Note that then (using the induction hypothesis $W_{k-1}\su \varphi_{k-1}(V_{k-1})$) we have
\[x_1\dotsm x_n \cdot (x_1^m+\dots+x_n^m)\cdot x_1^{2d_1}\dotsm x_{j-1}^{2d_{j-1}}x_{j}^{2d_{j}-2}x_{j+1}^{2d_{j+1}}\dotsm x_n^{2d_n}\in W_{k-1}\su \varphi_{k-1}(V_{k-1}).\]
Hence there exists a polynomial $P^*\in V_{k-1}$ of degree $n+2k-5$ whose homogeneous degree $n+2k-5$ part is $\varphi_{k-1}(P^*)=x_1\dotsm x_n \cdot (x_1^m+\dots+x_n^m)\cdot x_1^{2d_1}\dotsm x_{j-1}^{2d_{j-1}}x_{j}^{2d_{j}-2}x_{j+1}^{2d_{j+1}}\dotsm x_n^{2d_n}$. By Claim \ref{claim-multiplication-increase-k} we have $x_j(x_j-1)\cdot P^*\in V_k$, and note that the homogeneous degree  $n+2k-3$ part of $x_j(x_j-1)\cdot P^*$ is precisely $x_j^2\cdot \varphi_{k-1}(P^*)= x_1\dotsm x_n \cdot (x_1^m+\dots+x_n^m)\cdot x_1^{2d_1}\dotsm x_n^{2d_n}$. Hence $\varphi_k$ maps the polynomial $x_j(x_j-1)\cdot P^*\in V_k$ to $x_1\dotsm x_n \cdot (x_1^m+\dots+x_n^m)\cdot x_1^{2d_1}\dotsm x_n^{2d_n}$. This establishes that $W_k'\su \varphi_k(V_k)$.

Thus, $\varphi_k(V_k)\cap W_k$ is a subspace of $W_k$ with $W_{k}'\su \varphi_k(V_k)\cap W_k\su W_k$. Now, Proposition \ref{propo-key} states that there exists a polynomial $P\in V_k$ with $\varphi_k(P)\in W_k$ and $\varphi_k(P)\not\in W_k'$. Hence $\varphi_k(V_k)\cap W_k\neq W_k'$. Since $\dim W_k'=\dim W_k-1$, this implies that $\varphi_k(V_k)\cap W_k= W_k$ and consequently $W_k\su \varphi_k(V_k)$.
\end{proof}

\begin{corollary}\label{coro-phi-injective}
Let $k\geq 2$ and $n\geq 2k-3$. Then $\varphi_k(V_k)=W_k$, and the linear map $\varphi_k: V_k\to W_k$ is an isomorphism.
\end{corollary}
\begin{proof}
Recall that $\dim V_k=M_{k-1}(n)=\dim W_k$ by Claim \ref{claim-dim-Vk} and Corollary \ref{coro-dim-Wk}. Furthermore recall that by Corollary \ref{coro-Wk-in-image} we have $W_k\su \varphi_k(V_k)$. Since $\dim \varphi_k(V_k)\leq \dim V_k=\dim W_k$, we must have $\varphi_k(V_k)=W_k$ and the linear map $\varphi_k: V_k\to W_k$ must be an isomorphism.
\end{proof}

Using Corollary \ref{coro-phi-injective}, we are now ready to prove Theorem \ref{thm-2}.

\begin{proof}[Proof of Theorem \ref{thm-2}]
Fix $k\geq 2$ and $n\geq 2k-3$, and let $Q\in \RR[x_1,\dots,x_n]$ be a polynomial having zeroes of multiplicity at least $k$ at all points in $\{0,1\}^n\sm \{(0,\dots,0)\}$, and such that $Q$ does not have a zero of multiplicity at least $k-1$ at $(0,\dots,0)$. We need to show that $\deg Q\geq n+2k-3$. Let $P$ be a polynomial as in Claim \ref{claim-reduced-exists}. Then we have $P\in U_k$ and $\deg P\leq \deg Q$. Furthermore, $P$ also has zeroes of multiplicity at least $k$ at all points in $\{0,1\}^n\sm \{(0,\dots,0)\}$, which implies that $P\in V_k$. Finally, $P$ does not have a zero of multiplicity at least $k-1$ at $(0,\dots,0)$, which in particular means that $P$ is not the zero polynomial. Recall that by Corollary \ref{coro-phi-injective} the map $\varphi_k: V_k\to W_k$ is an isomorphism (and in particular injective). Thus, we have $\varphi_k(P)\neq 0$, which means that the homogeneous degree $n+2k-3$ part of $P$ is non-zero. So we can conclude that $\deg Q\geq \deg P\geq n+2k-3$.

It remains to prove the second part of Theorem \ref{thm-2}. Fix $0\leq \ell\leq k-2$. It is not hard to construct a (high-degree) polynomial $Q$ having zeroes of multiplicity at least $k$ at all points in $\{0,1\}^n\sm \{(0,\dots,0)\}$, and such that $Q$ has a zero of multiplicity exactly $\ell$ at $(0,\dots,0)$, for example the polynomial $Q=x_1^\ell \cdot (x_1-1)^k \dotsm (x_n-1)^k $. By Claim \ref{claim-reduced-exists}, there exists a polynomial $P\in U_k$ which also has zeroes of multiplicity at least $k$ at all points in $\{0,1\}^n\sm \{(0,\dots,0)\}$, and a zero of multiplicity exactly $\ell$ at $(0,\dots,0)$ (recall that $\ell\leq k-2$). It remains to check that $\deg P=n+2k-3$. Since $P\in U_k$, we must have $\deg P\leq n+2k-3$. Furthermore, the first part of Theorem \ref{thm-2} proved above implies that $\deg P\geq n+2k-3$. Hence $\deg P= n+2k-3$, and the polynomial $P$ has all of the desired properties.
\end{proof}

\section{Proof of Proposition \ref{propo-key}}
\label{sect-proof-proposition}

In this section, we will prove Proposition \ref{propo-key} by constructing a polynomial $P\in V_k$ satisfying the desired conditions. Our polynomial $P$ will be a symmetric polynomial.

Recall that we can write any symmetric polynomial $R\in \RR[x_1,\dots,x_n]$ in terms of the power sum symmetric polynomials $x_1^m+\dots+x_n^m$ for $1\leq m\leq n$. More precisely, we can write $R$ as a linear combination of products of power sum symmetric polynomials, i.e.\ as a linear combination of terms of the form $(x_1^{m_1}+\dots+x_n^{m_1})\dotsm (x_1^{m_\ell}+\dots+x_n^{m_\ell})$ for positive integers $m_1,\dots,m_\ell\leq n$. Furthermore, for every symmetric polynomial $R\in \RR[x_1,\dots,x_n]$, such a representation is unique (up to reordering $m_1,\dots,m_\ell$ in the terms of this linear combination), and we can therefore refer to it as ``the representation of $R$ in terms of power sum symmetric polynomials''.

\begin{claim}\label{claim-odd-exponents}
Let $n\geq 1$. Suppose that $d_1,\dots,d_t$ are positive integers with $d_1+\dots+d_t\leq n$, such that exactly one of $d_1,\dots,d_t$ is odd. Then the symmetric polynomial
\[\sum_{\substack{i_1,\dots,i_t\in \{1,\dots,n\}\\ \textnormal{distinct}}}x_{i_1}^{d_1}\dotsm x_{i_t}^{d_t}\]
is a linear combination of terms of the form $(x_1^{m_1}+\dots+x_n^{m_1})\dotsm (x_1^{m_\ell}+\dots+x_n^{m_\ell})$ for $m_1,\dots,m_\ell\in \{1,\dots,n\}$ with $m_1+\dots+m_\ell=d_1+\dots+d_t$ and such that exactly one of $m_1,\dots,m_\ell$ is odd.
\end{claim}
\begin{proof}Let us prove the claim by induction on $t$. If $t=1$, then $d_1$ is odd and the symmetric polynomial $x_1^{d_1}+\dots+x_n^{d_1}$ is already of the desired form. So let us now assume that $t\geq 2$ and that the claim holds for $t-1$. As only one of $d_1,\dots,d_t$ is odd, we may assume without loss of generality that $d_t$ is even. Note that
\begin{multline*}
\sum_{\substack{i_1,\dots,i_t\in \{1,\dots,n\}\\ \textnormal{distinct}}}x_{i_1}^{d_1}\dotsm x_{i_t}^{d_t}=(x_1^{d_t}+\dots+x_n^{d_t})\cdot \sum_{\substack{i_1,\dots,i_{t-1}\in \{1,\dots,n\}\\ \textnormal{distinct}}}x_{i_1}^{d_1}\dotsm x_{i_{t-1}}^{d_{t-1}}\\
-\sum_{s=1}^{t-1}\,\sum_{\substack{i_1,\dots,i_{t-1}\in \{1,\dots,n\}\\ \textnormal{distinct}}}x_{i_1}^{d_1}\dotsm x_{i_{s-1}}^{d_{s-1}}x_{i_{s}}^{d_{s}+d_t} x_{i_{s+1}}^{d_{s+1}}\dotsm x_{i_{t-1}}^{d_{t-1}}.
\end{multline*}
As exactly one of $d_1,\dots,d_{t-1}$ is odd, we can apply the induction hypothesis to the sum $\sum x_{i_1}^{d_1}\dotsm x_{i_{t-1}}^{d_{t-1}}$. We can then conclude that $(x_1^{d_t}+\dots+x_n^{d_t})\cdot \sum x_{i_1}^{d_1}\dotsm x_{i_{t-1}}^{d_{t-1}}$ is a linear combination of terms of the form $(x_1^{d_t}+\dots+x_n^{d_t})\cdot (x_1^{m_1}+\dots+x_n^{m_1})\dotsm (x_1^{m_\ell}+\dots+x_n^{m_\ell})$ for positive integers $m_1,\dots,m_\ell$ with $d_t+m_1+\dots+m_\ell=d_1+\dots+d_t$ and such that exactly one of $d_t,m_1,\dots,m_\ell$ is odd. Furthermore, for each $s=1,\dots,t-1$ we can also apply the induction hypothesis to the sum $\sum x_{i_1}^{d_1}\dotsm x_{i_{s-1}}^{d_{s-1}}x_{i_{s}}^{d_{s}+d_t} x_{i_{s+1}}^{d_{s+1}}\dotsm x_{i_{t-1}}^{d_{t-1}}$ and we find that each of these sums is a linear combination of terms of the desired form as well.
\end{proof}

\begin{claim}\label{claim-symmetric}
Let $k\geq 2$ and let $n\geq 2k-3$. Let $P\in U_k$ be a symmetric polynomial and let $\overline{P}$ be the homogeneous degree $n+2k-3$ part of $P$. Suppose that the polynomial $\overline{P}$ is divisible by $x_1\dotsm x_n$. Then we have $\overline{P}\in W_k$. Furthermore, when we write $\overline{P}$ in terms of the basis of $W_k$ in (\ref{eq-generators-Wk}), then the coefficient of $x_1\dotsm x_n \cdot (x_1^{2k-3}+\dots+x_n^{2k-3})$ is the same as the coefficient of $x_1^{2k-3}+\dots+x_n^{2k-3}$ when writing the symmetric polynomial $\overline{P}/(x_1\dotsm x_n)$ in terms of power sum symmetric polynomials.
\end{claim}
\begin{proof}
Recall that $P\in U_k$ does not contain any monomials which are divisible by $x_{i_1}^2\dotsm x_{i_k}^2$ for some (not necessarily distinct) indices $i_1,\dots,i_k\in \{1,\dots,n\}$. Therefore any monomial of $P$ must be of the form $x_1^{2m_1+r_1}\dotsm x_n^{2m_n+r_n}$ with non-negative integers $(m_1,\dots,m_n)$ with $m_1+\dots+ m_n\leq k-1$ and $r_1,\dots,r_n\in \{0,1\}$. If $(2m_1+r_1)+\dots +(2m_n+r_n)=n+2k-3$, then this is only possible if $m_1+\dots+ m_n= k-1$ and $r_1+\dots+ r_n= n-1$. This means that any monomial of $P$ of degree $n+2k-3$ must have an odd exponent for exactly $n-1$ of the $n$ variables $x_1,\dots,x_n$.

Now let $R=\overline{P}/(x_1\dotsm x_n)$, and note that $R\in \RR[x_1,\dots,x_n]$ is homogeneous of degree $2k-3$. Furthermore, since $P$ is symmetric, the polynomials $\overline{P}$ and $R$ are also symmetric. We saw above that every monomial in $\overline{P}$ has odd exponents for exactly $n-1$ of the variables $x_1,\dots,x_n$ (and an even exponent for the remaining variable). Hence every monomial in $R$ has an odd exponent for exactly one variable. Thus, we can apply Claim \ref{claim-odd-exponents} to all different types of monomials appearing in $R$ (recalling that $n\geq 2k-3$). We obtain that $R$ is a linear combination of terms of the form $(x_1^{m_1}+\dots+x_n^{m_1})\dotsm (x_1^{m_\ell}+\dots+x_n^{m_\ell})$ for $m_1,\dots,m_\ell\in \{1,\dots,n\}$ with $m_1+\dots+m_\ell=2k-3$ and such that exactly one of $m_1,\dots,m_\ell$ is odd. By renaming $m_1,\dots,m_\ell$, we can always assume that $m_1$ is odd and $m_2,\dots,m_\ell$ are even.

Note that this representation of $R$ as a linear combination of such terms is the representation of $R$ in terms of power sum symmetric polynomials. Hence the coefficient $\lambda$ of $x_1^{2k-3}+\dots+x_n^{2k-3}$ when expressing $\overline{P}/(x_1\dotsm x_n)=R$ in terms of power sum symmetric polynomials is precisely the coefficient of $x_1^{2k-3}+\dots+x_n^{2k-3}$ when expressing $R$ as a linear combination as above.

By multiplying with $x_1\dotsm x_n$, we can now express $\overline{P}=x_1\dotsm x_n\cdot R$ as a linear combination of terms of the form $x_1\dotsm x_n\cdot (x_1^{m_1}+\dots+x_n^{m_1})\dotsm (x_1^{m_\ell}+\dots+x_n^{m_\ell})$ for $m_1,\dots,m_\ell\in \{1,\dots,n\}$ with $m_1+\dots+m_\ell=2k-3$ such that $m_1$ is odd and $m_2,\dots,m_\ell$ are even. Note that the only such term with $m_1=2k-3$ is the term $x_1\dotsm x_n \cdot (x_1^{2k-3}+\dots+x_n^{2k-3})$, and the coefficient of this term equals $\lambda$.

Note that each of these terms $x_1\dotsm x_n\cdot (x_1^{m_1}+\dots+x_n^{m_1})\dotsm (x_1^{m_\ell}+\dots+x_n^{m_\ell})$ (with $m_1+\dots+m_\ell=2k-3$ and such that $m_1$ is odd and $m_2,\dots,m_\ell$ are even) can be written as a linear combination of terms of the form (\ref{eq-generators-Wk}) with $m=m_1$ (indeed, when multiplying out $(x_1^{m_2}+\dots+x_n^{m_2})\dotsm (x_1^{m_\ell}+\dots+x_n^{m_\ell})$ all variables always appear with even exponents). All in all, this shows that $\overline{P}$ can be written as a linear combination of terms of the form (\ref{eq-generators-Wk}), so $\overline{P}\in W_k$. Furthermore, when writing $\overline{P}$ in terms of the basis of $W_k$ in (\ref{eq-generators-Wk}), the coefficient of $x_1\dotsm x_n \cdot (x_1^{2k-3}+\dots+x_n^{2k-3})$ equals $\lambda$, which we defined as the coefficient of  $x_1^{2k-3}+\dots+x_n^{2k-3}$ when expressing $\overline{P}/(x_1\dotsm x_n)$ in terms of power sum symmetric polynomials.
\end{proof}

Let us now prove Proposition \ref{propo-key}. The proof of the proposition depends on three lemmas, whose proofs we will postpone to the next two subsections.

\begin{proof}[Proof of Proposition \ref{propo-key}]
Fix $k\geq 2$ and $n\geq 2k-3$. In order to construct the desired polynomial $P\in V_k$, let us first define the symmetric polynomial 
\begin{equation}\label{eq-def-Q}
    Q(x_1,\dots,x_n)=(-1)^{(k-1)n}\cdot (x_1-1)^k\dotsm (x_n-1)^k.
\end{equation}
Note that $Q$ has zeroes of multiplicity at least $k$ at all points in $\{0,1\}^n\sm \{(0,\dots,0)\}$.

Let us now define $P\in U_k$ to be the unique polynomial in $U_k$ such that $Q-P$ has zeroes of multiplicity at least $k$ at all points in $\{0,1\}^n\sm \{(0,\dots,0)\}$ and a zero of multiplicity at least $k-1$ at $(0,\dots,0)$ (such a polynomial $P$ exists by Claim \ref{claim-reduced-exists} and it is unique by Corollary \ref{corollary-reduced-unique}). Then $P$ has zeroes of multiplicity at least $k$ at all points in $\{0,1\}^n\sm \{(0,\dots,0)\}$, and hence $P\in V_k$. Furthermore, since $Q$ is symmetric and $P$ is unique, $P$ must also be a symmetric polynomial (otherwise we could permute the variables in $P$ and obtain another polynomial with the desired properties, contradicting uniqueness).

Recall that $\varphi_k(P)$ is the homogeneous degree $n+2k-3$ part of $P$. We need to show that $\varphi_k(P)\in W_k$ and that when writing $\varphi_k(P)\in W_k$ in terms of the basis in (\ref{eq-generators-Wk}), then the coefficient of the basis element $x_1\dotsm x_n \cdot (x_1^{2k-3}+\dots+x_n^{2k-3})$ is non-zero.

We wish to apply Claim \ref{claim-symmetric} to $P$, but this requires the homogeneous degree $n+2k-3$ part $\varphi_k(P)$ of $P$ to be divisible by $x_1\dotsm x_n$.

\begin{lemma}\label{lemma-divisibility}
The polynomial $\varphi_k(P)$ is divisible by $x_1\dotsm x_n$.
\end{lemma}

We postpone the proof of Lemma \ref{lemma-divisibility} to Subsection \ref{subsect-proof-two-lemmas}. Assuming Lemma \ref{lemma-divisibility}, by Claim \ref{claim-symmetric} we obtain $\varphi_k(P)\in W_k$, which establishes the first of our two required properties for $P$. Furthermore, we also obtain from Claim \ref{claim-symmetric} that the desired coefficient of $\varphi_k(P)$ (when writing it in terms of the basis elements of $W_k$) is equal to the coefficient of $x_1^{2k-3}+\dots+x_n^{2k-3}$ when writing $\varphi_k(P)/(x_1\dotsm x_n)$ in terms of power sum symmetric polynomials. So it remains to prove that this coefficient of  $x_1^{2k-3}+\dots+x_n^{2k-3}$ is non-zero. The following lemma determines this coefficient.

\begin{lemma}\label{lemma-coefficient}
When writing $\varphi_k(P)/(x_1\dotsm x_n)$ in terms of power sum symmetric polynomials, the coefficient of $x_1^{2k-3}+\dots+x_n^{2k-3}$ is equal to
\[\sum_{(m_1,\dots,m_t)}(-1)^{t}\cdot \binom{k-1-m_1}{m_1-1}\binom{k-1-m_2}{m_2}\dots \binom{k-1-m_t}{m_t},\]
where the sum is over all sequences $(m_1,\dots,m_t)$ of positive integers with $m_1+\dots+m_t=k-1$.
\end{lemma}

We also postpone  the proof of Lemma \ref{lemma-coefficient} to Subsection \ref{subsect-proof-two-lemmas}. Surprisingly, the value of the sum in Lemma \ref{lemma-coefficient} is (up to its sign) equal to the Catalan number $C_{k-2}$; see the following lemma. Recall that the Catalan numbers $C_0, C_1, \dots$ are given by the explicit formula $C_n=\binom{2n}{n}/(n+1)$ for all $n\geq 0$.

\begin{lemma}\label{lemma-catalan}
For any $\ell\geq 1$, we have
\[\sum_{(m_1,\dots,m_t)}(-1)^{t}\cdot \binom{\ell-m_1}{m_1-1}\binom{\ell-m_2}{m_2}\dots \binom{\ell-m_t}{m_t}=(-1)^\ell C_{\ell-1},\]
where the sum is over all sequences $(m_1,\dots,m_t)$ of positive integers with $m_1+\dots+m_t=\ell$.
\end{lemma}

We postpone  the proof of Lemma \ref{lemma-catalan} to Subsection \ref{subsect-proof-last-lemma}. Applying the lemma with $\ell=k-1$, we see that the sum in Lemma \ref{lemma-coefficient} equals $(-1)^{k-1} C_{k-2}=(-1)^{k-1}\binom{2k-4}{k-2}/(k-1)\neq 0$. Thus,  the coefficient of $x_1^{2k-3}+\dots+x_n^{2k-3}$ when writing $\varphi_k(P)/(x_1\dotsm x_n)$ in terms of power sum symmetric polynomials is indeed non-zero, finishing the proof of the Proposition \ref{propo-key}.
\end{proof}

\subsection{Proof of Lemmas \ref{lemma-divisibility} and \ref{lemma-coefficient}}
\label{subsect-proof-two-lemmas}

In this subsection, we prove Lemmas \ref{lemma-divisibility} and \ref{lemma-coefficient} by analyzing the polynomial $P$ and its homogeneous degree $n+2k-3$ part $\varphi_k(P)$. As before, let $k\geq 2$ and $n\geq 2k-3$ be fixed.

In order to calculate $\varphi_k(P)$, let us first rewrite the definition of the polynomial $Q$ in (\ref{eq-def-Q}) as
\[Q(x_1,\dots,x_n)= \prod_{i=1}^{n}\left((-1)^{k-1}\cdot (x_i-1)^{k}\right).\]
For $i=1,\dots,n$, let us now write $(-1)^{k-1}\cdot (x_i-1)^{k}$ as a linear combination of $\allowbreak (x_i-1),\allowbreak x_i(x_i-1),\allowbreak x_i(x_i-1)^2,\allowbreak x_i^2(x_i-1)^2,\allowbreak x_i^2(x_i-1)^3,\dots$, as in the following claim.

\begin{claim}\label{claim-power-xi-1}
For any $\ell\geq 1$ and $i\in \{1,\dots,n\}$, we have
\[(-1)^{\ell-1}\cdot (x_i-1)^{\ell}=\sum_{m=0}^{\lfloor (\ell-1)/2\rfloor}\binom{\ell-1-m}{m} x_i^m(x_i-1)^{m+1}-\sum_{m=1}^{\lfloor \ell/2\rfloor}\binom{\ell-1-m}{m-1} x_i^m(x_i-1)^{m}.\]
\end{claim}
\begin{proof}
First, note that we may equivalently write the desired equation as
\[(-1)^{\ell-1}\cdot (x_i-1)^{\ell}=\sum_{m=0}^{\infty}\binom{\ell-1-m}{m} x_i^m(x_i-1)^{m+1}-\sum_{m=1}^{\infty}\binom{\ell-1-m}{m-1} x_i^m(x_i-1)^{m},\]
since for all terms appearing in the infinite sums that did not appear in the original equation, the respective binomial coefficients are zero.

Let us now prove this equation by induction on $\ell$. For $\ell=1$, the equation is easy to check. Let us now assume it is true for some $\ell\geq 1$, and let us check it for $\ell+1$. We have
\begin{align*}
(-1)^{\ell}(x_i-1)^{\ell+1}&= -(x_i-1)\sum_{m=0}^{\infty}\binom{\ell-1-m}{m} x_i^m(x_i-1)^{m+1}+(x_i-1)\sum_{m=1}^{\infty}\binom{\ell-1-m}{m-1} x_i^m(x_i-1)^{m}\\
&=-\sum_{m=0}^{\infty}\binom{\ell-1-m}{m} x_i^{m+1}(x_i-1)^{m+1}+\sum_{m=0}^{\infty}\binom{\ell-1-m}{m} x_i^m(x_i-1)^{m+1}\\
&\pushright{+\sum_{m=1}^{\infty}\binom{\ell-1-m}{m-1} x_i^m(x_i-1)^{m+1}}\\
&=-\sum_{m=1}^{\infty}\binom{\ell-m}{m-1} x_i^{m}(x_i-1)^{m}+\sum_{m=0}^{\infty}\binom{\ell-m}{m} x_i^m(x_i-1)^{m+1},
\end{align*}
as desired. This proves the claim.
\end{proof}

From Claim \ref{claim-power-xi-1}, we now obtain that $Q(x_1,\dots,x_n)$ equals
\begin{equation}\label{eq-expression-Q}
    \prod_{i=1}^{n}\left(\sum_{m=0}^{\lfloor (k-1)/2\rfloor}\binom{k-1-m}{m} x_i^m(x_i-1)^{m+1}-\sum_{m=1}^{\lfloor k/2\rfloor}\binom{k-1-m}{m-1} x_i^m(x_i-1)^{m}\right).
\end{equation}

Using this equation, we can calculate the polynomial $P$. Recall that $P$ is the unique polynomial in $U_k$ such that $Q-P$ has zeroes of multiplicity at least $k$ at all points in $\{0,1\}^n\sm \{(0,\dots,0)\}$ and a zero of multiplicity at least $k-1$ at $(0,\dots,0)$.

When expanding the product over all $i$ in (\ref{eq-expression-Q}), each of the terms that we obtain is a product of factors of the form $x_i^{m_i}(x_i-1)^{m_i+1}$ or $x_i^{m_i}(x_i-1)^{m_i}$ for each $i=1,\dots,n$ (with some coefficient given as the product of certain binomial coefficients). Note that each such term is divisible by the product $x_1^{m_1}(x_1-1)^{m_1}\dotsm x_n^{m_n}(x_n-1)^{m_n}$, and if $m_1+\dots+m_n\geq k$ this means that the term has zeroes of multiplicity at least $k$ at all points in $\{0,1\}^n$. We can omit all such terms from $Q$ without violating the defining property of $P\in U_k$ (i.e.\ the property that $Q-P$ has zeroes of multiplicity at least $k$ at all points in $\{0,1\}^n\sm \{(0,\dots,0)\}$ and a zero of multiplicity at least $k-1$ at $(0,\dots,0)$).

So let us now imagine that we expand the product over $i$ in (\ref{eq-expression-Q}), but we only keep the terms which are products of factors of the form $x_i^{m_i}(x_i-1)^{m_i+1}$ or $x_i^{m_i}(x_i-1)^{m_i}$ for each $i=1,\dots,n$ such that $m_1+\dots+m_n\leq k-1$. Note that all of these terms have degree at most $n+2k-2$. Furthermore, any such term of degree $n+2k-2$ must be of the form $x_1^{m_1}(x_1-1)^{m_1+1}\dotsm x_n^{m_n}(x_n-1)^{m_n+1}=(x_1-1)\dotsm (x_n-1)\cdot x_1^{m_1}(x_1-1)^{m_1}\dotsm x_n^{m_n}(x_n-1)^{m_n}$ with  $m_1+\dots+m_n=k-1$. Hence any such term of degree $n+2k-2$ has zeroes of multiplicity at least $k$ at all points in $\{0,1\}^n\sm \{(0,\dots,0)\}$ and a zero of multiplicity at least $k-1$ at $(0,\dots,0)$. We can also omit all of these terms from $Q$ without violating the defining property of $P\in U_k$.

Now, let $Q^*$ be the polynomial consisting only of the remaining terms in the expansion of the product over $i$ in (\ref{eq-expression-Q}), i.e.\ of the terms which are products of factors of the form $x_i^{m_i}(x_i-1)^{m_i+1}$ or $x_i^{m_i}(x_i-1)^{m_i}$ for each $i=1,\dots,n$ such that $m_1+\dots+m_n\leq k-1$ and which have degree at most $n+2k-3$. Then $Q^*-P$ has zeroes of multiplicity at least $k$ at all points in $\{0,1\}^n\sm \{(0,\dots,0)\}$ and a zero of multiplicity at least $k-1$ at $(0,\dots,0)$. However, the polynomial $Q^*$ is also $k$-reduced, so we have $Q^*\in U_k$. By applying Corollary \ref{corollary-reduced-unique} to the polynomial $Q^*$, we can therefore conclude that $Q^*=P$ (since $P,Q^*\in U_k$ both satisfy the conditions in  Corollary \ref{corollary-reduced-unique}).

In other words, we just showed that $P$ can be obtained by expanding the product over $i$ in (\ref{eq-expression-Q}), but only keeping the terms which are products of factors of the form $x_i^{m_i}(x_i-1)^{m_i+1}$ or $x_i^{m_i}(x_i-1)^{m_i}$ for each $i=1,\dots,n$ such that $m_1+\dots+m_n\leq k-1$ and which have degree at most $n+2k-3$. Now, the homogeneous degree $n+2k-3$ part $\varphi_k(P)$ of $P$ can be calculated by only considering the products of degree exactly $n+2k-3$ (and by only taking the homogeneous degree $n+2k-3$ part of these products). In other words, $\varphi_k(P)$ is obtained by expanding
\begin{equation}\label{eq-expression-phi-P}
    \prod_{i=1}^{n}\left(\sum_{m=0}^{\lfloor (k-1)/2\rfloor}\binom{k-1-m}{m} x_i^{2m+1}-\sum_{m=1}^{\lfloor k/2\rfloor}\binom{k-1-m}{m-1} x_i^{2m}\right),
\end{equation}
but only keeping the terms which are products of factors of the form $x_i^{2m_i+1}$ or $x_i^{2m_i}$ for each $i=1,\dots,n$ such that $m_1+\dots+m_n\leq k-1$ and which have degree exactly $n+2k-3$. Note that such a term can only have degree $n+2k-3$ if it consists of $n-1$ factors of the form $x_i^{2m_i+1}$ and one factor of the form $x_i^{2m_i}$ (with $m_1+\dots+m_n= k-1$). Hence $\varphi_k(P)$ can be obtained by expanding (\ref{eq-expression-phi-P}), but only keeping the monomials of degree exactly $n+2k-3$ in which exactly one variable has an even exponent.

To simplify notation, let us define $a_0,\dots,a_{k-1}$ to be the coefficients such that (\ref{eq-expression-phi-P}) can be written as $\prod_{i=1}^n(a_{k-1}x_i^{k}+a_{k-2}x_i^{k-1}+\dots+a_0x_i)$. Furthermore, for $d\geq k$, let us define $a_d=0$. Note that then for all $m\geq 0$ we have $a_{2m}=\binom{k-1-m}{m}$, and for all $m\geq 1$ we have $a_{2m-1}=-\binom{k-1-m}{m-1}$. Also note that $a_0=1$. Using this notation, we now obtain that
\begin{equation}\label{eq-phi-P-formal}
\varphi_k(P)=\sum_{(d_1,\dots,d_n)} a_{d_1}\dotsm a_{d_n}\cdot x_{1}^{d_1+1}\dotsm x_{n}^{d_n+1},
\end{equation}
where the sum is over all sequences $(d_1,\dots,d_n)$ of non-negative integers with $(d_1+1)+\dots+(d_n+1)=n+2k-3$ such that exactly one of $d_1+1,\dots,d_n+1$ is even. These conditions are equivalent to demanding that $d_1+\dots+d_n=2k-3$ and exactly one of $d_1,\dots,d_n$ is odd. Now, Lemma \ref{lemma-divisibility} follows easily.

\begin{proof}[Proof of Lemma \ref{lemma-divisibility}]
All monomials appearing in (\ref{eq-phi-P-formal}) are divisible by $x_1\dotsm x_n$, and hence $\varphi_k(P)$ is divisible by $x_1\dotsm x_n$, as desired.
\end{proof}

From (\ref{eq-phi-P-formal}), we now obtain that
\[\varphi_k(P)/(x_1\dotsm x_n)=\sum_{(d_1,\dots,d_n)} a_{d_1}\dotsm a_{d_n}\cdot x_{1}^{d_1}\dotsm x_{n}^{d_n},\]
where the sum is over all sequences $(d_1,\dots,d_n)$ of non-negative integers with $d_1+\dots+d_n=2k-3$ such that exactly one of $d_1,\dots,d_n$ is odd. Recalling that $a_0=1$, we can rewrite this equation as
\begin{equation}\label{eq-hom-deg-2k-3}
   \varphi_k(P)/(x_1\dotsm x_n)=\sum_{(d_1,\dots,d_t)}\left(a_{d_1}\dotsm a_{d_t}\sum_{1\leq i_1<\dots<i_t\leq n}x_{i_1}^{d_1}\dotsm x_{i_t}^{d_t}\right),
\end{equation}
where the sum is over all sequences $(d_1,\dots,d_t)$ of positive integers with $d_1+\dots+d_t=2k-3$ such that exactly one of $d_1,\dots,d_t$ is odd. Note that when we consider all $t!$ permutations of one fixed such sequence $(d_1,\dots,d_t)$ (some of which may be equal to the original sequence), then the resulting part of the sum in (\ref{eq-hom-deg-2k-3}) is the symmetric polynomial $a_{d_1}\dotsm a_{d_t}\sum_{\substack{i_1,\dots, i_t\\\textnormal{distinct}}}x_{i_1}^{d_1}\dotsm x_{i_t}^{d_t}$. By averaging over all permutations of $(d_1,\dots,d_t)$, we can therefore rewrite (\ref{eq-hom-deg-2k-3}) as
\begin{equation}\label{eq-hom-deg-2k-3-symm}
\varphi_k(P)/(x_1\dotsm x_n)=\sum_{(d_1,\dots,d_t)}\left(\frac{a_{d_1}\dotsm a_{d_t}}{t!}\sum_{\substack{i_1,\dots,i_t\in \{1,\dots,n\}\\ \textnormal{distinct}}}x_{i_1}^{d_1}\dotsm x_{i_t}^{d_t}\right),
\end{equation}
where the sum is again over all sequences $(d_1,\dots,d_t)$ of positive integers with $d_1+\dots+d_t=2k-3$ such that exactly one of $d_1,\dots,d_t$ is odd.

In order to prove Lemma \ref{lemma-coefficient}, we need to find the coefficient of $x_1^{2k-3}+\dots+x_n^{2k-3}$ when expressing $\varphi_k(P)/(x_1\dotsm x_n)$ in terms of power sum symmetric polynomials. We will use the following claim, which can be derived from more general statements in the theory of symmetric polynomials. For the reader's convenience we provide a simple self-contained proof.

\begin{claim}\label{claim-coeff-newton}
For any sequence $(d_1,\dots,d_t)$ of positive integers with $d_1+\dots+d_t\leq n$, when expressing
\[\sum_{\substack{i_1,\dots,i_t\in \{1,\dots,n\}\\ \textnormal{distinct}}}x_{i_1}^{d_1}\dotsm x_{i_t}^{d_t}\]
in terms of power sum symmetric polynomials, the coefficient of $x_1^{d_1+\dots+d_t}+\dots+x_n^{d_1+\dots+d_t}$ equals $t!\cdot (-1)^{t-1}/t$.
\end{claim}
\begin{proof}For $t=1$, the claim is trivially true. Let us now assume that $t\geq 2$, and that we already proved the claim for $t-1$. Note that
\begin{multline*}
\sum_{\substack{i_1,\dots,i_t\in \{1,\dots,n\}\\ \textnormal{distinct}}}x_{i_1}^{d_1}\dotsm x_{i_t}^{d_t}=(x_1^{d_t}+\dots+x_n^{d_t})\cdot \sum_{\substack{i_1,\dots,i_{t-1}\in \{1,\dots,n\}\\ \textnormal{distinct}}}x_{i_1}^{d_1}\dotsm x_{i_{t-1}}^{d_{t-1}}\\
-\sum_{s=1}^{t-1}\,\sum_{\substack{i_1,\dots,i_{t-1}\in \{1,\dots,n\}\\ \textnormal{distinct}}}x_{i_1}^{d_1}\dotsm x_{i_{s-1}}^{d_{s-1}}x_{i_{s}}^{d_{s}+d_t} x_{i_{s+1}}^{d_{s+1}}\dotsm x_{i_{t-1}}^{d_{t-1}}.
\end{multline*}
Let us now imagine that we express the sums on the left-hand side in terms of power sum symmetric polynomials. The terms contributed from the first part (before the minus sign), all contain a factor $x_1^{d_t}+\dots+x_n^{d_t}$, and in particular this first part does not contribute any $x_1^{d_1+\dots+d_t}+\dots+x_n^{d_1+\dots+d_t}$ terms. For the second part (after the minus sign), the coefficient of $x_1^{d_1+\dots+d_t}+\dots+x_n^{d_1+\dots+d_t}$ is by the induction hypothesis equal to $-\sum_{s=1}^{t-1}(t-1)!\cdot (-1)^{t-2}/(t-1)=(t-1)!\cdot (-1)^{t-1}=t!\cdot (-1)^{t-1}/t$. This finishes the proof of the claim.
\end{proof}

\begin{proof}[Proof of Lemma \ref{lemma-coefficient}]
Let $Y$ denote the coefficient of $x_1^{2k-3}+\dots+x_n^{2k-3}$ when writing $\varphi_k(P)/(x_1\dotsm x_n)$ in terms of power sum symmetric polynomials. We need to prove that $Y$ equals the sum in the statement of Lemma \ref{lemma-coefficient}.

Recalling that we assumed $n\geq 2k-3$, we can apply Claim \ref{claim-coeff-newton} to the terms on the right-hand side of (\ref{eq-hom-deg-2k-3-symm}), and obtain
\[Y=\sum_{(d_1,\dots,d_t)}\frac{a_{d_1}\dotsm a_{d_t}}{t!}\cdot \frac{t!\cdot (-1)^{t-1}}{t}=\sum_{(d_1,\dots,d_t)}a_{d_1}\dotsm a_{d_t}\cdot \frac{(-1)^{t-1}}{t},\]
where the sums are over all sequences $(d_1,\dots,d_t)$ of positive integers with $d_1+\dots+d_t=2k-3$ such that exactly one of $d_1,\dots,d_t$ is odd. Note that all $t$ cyclic permutations of one fixed such sequence $(d_1,\dots,d_t)$ contribute the same amount to the sum above, and for exactly one of these permutations $d_1$ is odd. Hence we can conclude that
\[Y=\sum_{(d_1,\dots,d_t)}t\cdot a_{d_1}\dotsm a_{d_t}\cdot \frac{(-1)^{t-1}}{t}=\sum_{(d_1,\dots,d_t)}(-1)^{t-1}a_{d_1}\dotsm a_{d_t},\]
where this time the sums are over all sequences $(d_1,\dots,d_t)$ of positive integers with $d_1+\dots+d_t=2k-3$ such that $d_1$ is odd and $d_2,\dots,d_t$ are even.

Let us now change variables, writing $d_1=2m_1-1$ and $d_j=2m_j$ for $j=2,\dots,t$. Then we obtain
\begin{multline*}
Y=\sum_{(m_1,\dots,m_t)}(-1)^{t-1}a_{2m_1-1}a_{2m_2}\dotsm  a_{2m_t}=\sum_{(m_1,\dots,m_t)}(-1)^{t}\cdot (-a_{2m_1-1})\cdot a_{2m_2}\dotsm a_{2m_t}\\
=\sum_{(m_1,\dots,m_t)}(-1)^{t}\cdot \binom{k-1-m_1}{m_1-1}\binom{k-1-m_2}{m_2}\dots \binom{k-1-m_t}{m_t},
\end{multline*}
where the sums are over all sequences $(m_1,\dots,m_t)$ of positive integers with $m_1+\dots+m_t=k-1$. This proves Lemma \ref{lemma-coefficient}.
\end{proof}

\subsection{Proof of Lemma \ref{lemma-catalan}}
\label{subsect-proof-last-lemma}

In this subsection, we prove Lemma \ref{lemma-catalan}. We remark that after an earlier version of this paper was posted, alternative proofs were found, by Ekhad and Zeilberger \cite{ekhad-zeilberger} and by Carde \cite{carde}.

We will use the formula for Catalan numbers stated in the following claim. This formula follows from work of Riordan \cite{riordan} and is also stated as Theorem 12.1 in a book on Catalan numbers by Koshy \cite{koshy}. It can also be proved from a simple bijection argument, as shown in \cite{ekhad-zeilberger}. For the reader's convenience we give a self-contained proof here.
\begin{claim}For all $s\geq 1$, we have
\begin{equation}\label{eq-formula-catalan}
    \sum_{i=0}^{s}(-1)^i C_i\cdot \binom{i+1}{s-i}=0.
\end{equation}
\end{claim}
\begin{proof}We prove the desired formula by induction on $s$. For $s=1$ and $s=2$, the formula is easy to check. So let us now assume that $s\geq 3$ and that the formula is true for $s-1$ and $s-2$.

Recall that for $i\geq 0$ we have $C_i=\binom{2i}{i}/(i+1)$, which implies that $C_i=C_{i-1}\cdot 2(2i-1)/(i+1)$ for $i\geq 1$. Consequently, for all $i\geq 1$ we have
\begin{align*}
C_i \cdot\binom{i+1}{s-i}&=2C_{i-1}\cdot(2i-1)\cdot \frac{1}{i+1}\binom{i+1}{s-i}\\
&=2C_{i-1} \cdot\left(\frac{2s-4}{s+1}(s-i)+\frac{2s-1}{s+1}(2i+1-s)\right)\cdot \frac{1}{i+1}\binom{i+1}{s-i}\\
&=\frac{4s-8}{s+1}\cdot C_{i-1}\cdot \frac{s-i}{i+1}\binom{i+1}{s-i}+\frac{4s-2}{s+1}\cdot C_{i-1}\cdot \frac{2i+1-s}{i+1}\binom{i+1}{s-i}\\
&=\frac{4s-8}{s+1}\cdot C_{i-1}\cdot \binom{i}{s-i-1}+\frac{4s-2}{s+1}\cdot C_{i-1}\cdot \binom{i}{s-i}.
\end{align*}
Now, recalling that $s\geq 3$, we obtain that
\begin{align*}
    \sum_{i=0}^{s}(-1)^i C_i\cdot \binom{i+1}{s-i}&=\sum_{i=1}^{\infty}(-1)^i C_i\cdot \binom{i+1}{s-i}\\
    &=-\frac{4s-8}{s+1}\sum_{i=1}^{\infty}(-1)^{i-1}C_{i-1}\cdot \binom{i}{s-1-i}-\frac{4s-2}{s+1}\sum_{i=1}^{\infty}(-1)^{i-1}C_{i-1}\cdot \binom{i}{s-i}\\
    &=-\frac{4s-8}{s+1}\sum_{i=0}^{\infty}(-1)^{i}C_{i}\cdot \binom{i+1}{s-2-i}-\frac{4s-2}{s+1}\sum_{i=0}^{\infty}(-1)^{i}C_{i}\cdot \binom{i+1}{s-1-i}\\
    &=-\frac{4s-8}{s+1}\sum_{i=0}^{s-2}(-1)^{i}C_{i}\cdot \binom{i+1}{s-2-i}-\frac{4s-2}{s+1}\sum_{i=0}^{s-1}(-1)^{i}C_{i}\cdot \binom{i+1}{s-1-i}=0,
\end{align*}
using the induction hypothesis for $s-2$ and $s-1$ in the last step.
\end{proof}

We will also use the following well-known formula for binomial coefficients (which is easy to prove, for example by double-counting): for all non-negative integers $m$, $n$ and $s$, we have
\begin{equation}\label{eq-formula-binom-coeff}
    \binom{n+s+1}{m-s}=\sum_{j= s}^\infty\binom{s+1}{j-s}\binom{n}{m-j}.
\end{equation}

From (\ref{eq-formula-catalan}) and (\ref{eq-formula-binom-coeff}) we can derive the following statement, which will be used in the proof of Lemma \ref{lemma-catalan}.

\begin{claim}\label{claim-expression-via-b}
For any non-negative integers $m$, $n$ and $s$, we have
\[\binom{n}{m}=\sum_{j=0}^{s}(-1)^jC_j\cdot \binom{n+j+1}{m-j}-\sum_{j=s+1}^{2s+1}\left(\sum_{i=0}^s (-1)^iC_i\cdot \binom{i+1}{j-i}\right)\binom{n}{m-j}.\]
\end{claim}

\begin{proof}
First, note that for any $j>2s+1$, we have $\binom{i+1}{j-i}=0$ for all $i=0,\dots,s$. We can therefore equivalently write the desired equation as
\[\binom{n}{m}=\sum_{j=0}^{s}(-1)^jC_j\cdot \binom{n+j+1}{m-j}-\sum_{j=s+1}^{\infty}\left(\sum_{i=0}^s (-1)^iC_i\cdot \binom{i+1}{j-i}\right)\binom{n}{m-j}.\]
We prove this equation by induction on $s$. Note that for $s=0$, it simply states that (recall that $C_0=1$)
\[\binom{n}{m}=\binom{n+1}{m}-\binom{n}{m-1},\]
which is true. Now let us assume that $s\geq 1$ and that the desired equation is true for $s-1$. Then
\begin{align*}
\binom{n}{m}&=\sum_{j=0}^{s-1}(-1)^jC_j\cdot \binom{n+j+1}{m-j}-\sum_{j=s}^{\infty}\left(\sum_{i=0}^{s-1} (-1)^iC_i\cdot \binom{i+1}{j-i}\right)\binom{n}{m-j}\\
&=\sum_{j=0}^{s}(-1)^jC_j\cdot \binom{n+j+1}{m-j}-(-1)^sC_s\cdot \binom{n+s+1}{m-s}-\sum_{j=s}^{\infty}\left(\sum_{i=0}^{s-1} (-1)^iC_i\cdot \binom{i+1}{j-i}\right)\binom{n}{m-j}\\
&=\sum_{j=0}^{s}(-1)^jC_j\cdot \binom{n+j+1}{m-j}-\sum_{j=s}^{\infty}\left((-1)^sC_s\cdot\binom{s+1}{j-s}+ \sum_{i=0}^{s-1} (-1)^iC_i\cdot \binom{i+1}{j-i}\right)\binom{n}{m-j}\\
&=\sum_{j=0}^{s}(-1)^jC_j\cdot \binom{n+j+1}{m-j}-\sum_{j=s}^{\infty}\left(\sum_{i=0}^{s} (-1)^iC_i\cdot \binom{i+1}{j-i}\right)\binom{n}{m-j}\\
&=\sum_{j=0}^{s}(-1)^jC_j\cdot \binom{n+j+1}{m-j}-\sum_{j=s+1}^{\infty}\left(\sum_{i=0}^{s} (-1)^iC_i\cdot \binom{i+1}{j-i}\right)\binom{n}{m-j},
\end{align*}
where for the third equation we used (\ref{eq-formula-binom-coeff}), and in the last equation we used (\ref{eq-formula-catalan}). This finishes the proof of the claim.
\end{proof}
 
Using Claim \ref{claim-expression-via-b}, we now prove Lemma \ref{lemma-catalan}.

\begin{proof}[Proof of Lemma \ref{lemma-catalan}]
Fix $\ell\geq 1$. For convenience, let us denote the left-hand side of the equation in Lemma \ref{lemma-catalan} by $Z$. We then need to show that $Z=(-1)^\ell C_{\ell-1}$.

We have
\[Z=\sum_{\substack{(m_1,\dots,m_t)\\ m_1,\dots,m_t>0\\m_1+\dots+m_t=\ell}}(-1)^{t}\cdot \binom{\ell-m_1}{m_1-1}\binom{\ell-m_2}{m_2}\dots \binom{\ell-m_t}{m_t}.\]
Applying Claim \ref{claim-expression-via-b} to the first binomial coefficient in the product (with $n=\ell-m_1$ and $m=m_1-1$ and $s=\ell-1$), we obtain
\begin{multline*}
Z=\sum_{j=0}^{\ell-1}\sum_{\substack{(m_1,\dots,m_t)\\ m_1,\dots,m_t>0\\m_1+\dots+m_t=\ell}}(-1)^{t}\cdot (-1)^jC_j\binom{\ell-m_1+j+1}{m_1-1-j}\binom{\ell-m_2}{m_2}\dots \binom{\ell-m_t}{m_t}\\
-\sum_{j=\ell}^{2\ell-1}\sum_{\substack{(m_1,\dots,m_t)\\ m_1,\dots,m_t>0\\m_1+\dots+m_t=\ell}}(-1)^{t}\cdot\left(\sum_{i=0}^{\ell-1} (-1)^iC_i\binom{i+1}{j-i}\right) \binom{\ell-m_1}{m_1-1-j}\binom{\ell-m_2}{m_2}\dots \binom{\ell-m_t}{m_t}.
\end{multline*}
Note that the sum after the minus sign in the above equation is zero. Indeed, for all terms appearing in this sum we have $m_1\leq \ell\leq j$, so the binomial coefficient $\binom{\ell-m_1}{m_1-1-j}$ is zero. Hence
\[Z=\sum_{j=0}^{\ell-1}\sum_{\substack{(m_1,\dots,m_t)\\ m_1,\dots,m_t>0\\m_1+\dots+m_t=\ell\\t\geq 1}}(-1)^{j+t} C_j\binom{\ell-m_1+j+1}{m_1-j-1}\binom{\ell-m_2}{m_2}\dots \binom{\ell-m_t}{m_t}.\]
With a change of variables, replacing $m_1-j-1$ by $m_1$, we can rewrite this as
\[Z=\sum_{j=0}^{\ell-1}\sum_{\substack{(m_1,\dots,m_t)\\ m_1\geq 0,\, m_2,\dots,m_t>0\\m_1+\dots+m_t=\ell-j-1\\t\geq 1}}(-1)^{j+t} C_j\binom{\ell-m_1}{m_1}\binom{\ell-m_2}{m_2}\dots \binom{\ell-m_t}{m_t}\]
(here, a priori the condition on $m_1$ would be $m_1>-j-1$, but since the binomial coefficient $\binom{\ell-m_1}{m_1}$ vanishes for negative $m_1$, we can instead take the condition $m_1\geq 0$). Note that in the above equation, the contribution for $j=\ell-1$ is just $(-1)^{\ell-1+1}C_{\ell-1}\binom{\ell-0}{0}=(-1)^{\ell}C_{\ell-1}$. Thus,
\[Z=(-1)^{\ell}C_{\ell-1}+\sum_{j=0}^{\ell-2}\sum_{\substack{(m_1,\dots,m_t)\\ m_1\geq 0,\, m_2,\dots,m_t>0\\m_1+\dots+m_t=\ell-j-1}}(-1)^{j+t}C_j\binom{\ell-m_1}{m_1}\binom{\ell-m_2}{m_2}\dots \binom{\ell-m_t}{m_t}.\]
Now, in order to prove the lemma, it suffices to show that for any fixed $j\in \{0,\dots,\ell-2\}$ we have
\begin{equation}\label{eq-cancellation}
\sum_{\substack{(m_1,\dots,m_t)\\ m_1\geq 0,\, m_2,\dots,m_t>0\\m_1+\dots+m_t=\ell-j-1}}(-1)^{t}\cdot \binom{\ell-m_1}{m_1}\binom{\ell-m_2}{m_2}\dots \binom{\ell-m_t}{m_t}=0.
\end{equation}
Indeed, by distinguishing whether $m_1$ is positive or zero (and noting that for $m_1=0$ we have $\binom{\ell-m_1}{m_1}=\binom{\ell-0}{0}=1$), we can rewrite the left-hand size of (\ref{eq-cancellation}) as
\[\sum_{\substack{(m_1,\dots,m_t)\\ m_1, m_2,\dots,m_t>0\\m_1+\dots+m_t=\ell-j-1}}(-1)^{t}\cdot \binom{\ell-m_1}{m_1}\dots \binom{\ell-m_t}{m_t}+\sum_{\substack{(m_2,\dots,m_t)\\ m_2,\dots,m_t>0\\m_2+\dots+m_t=\ell-j-1}}(-1)^{t}\cdot \binom{\ell-m_2}{m_2}\dots \binom{\ell-m_t}{m_t}=0,\]
where the equality follows because the exact same summands appear in both sums, but with opposite signs. This proves  (\ref{eq-cancellation}), finishing the proof of Lemma \ref{lemma-catalan}.
\end{proof}

\section{Concluding remarks}

\subsection{Clifton and Huang's hyperplane problem}\label{subsec:clifton-huang}

As mentioned in the introduction, Clifton and Huang \cite{clifton-huang} studied the minimum size of a collection of hyperplanes in $\RR^n$ such that every point in $\{0,1\}^n\sm \{(0,\dots,0)\}$ is covered by at least $k$ of these hyperplanes, but no hyperplane contains $(0,\dots,0)$, where $k\geq 2$ is fixed and $n$ is large with respect to $k$. While Theorem \ref{thm-1} improves their lower bound for this problem to $n+2k-3$, the best known upper bound is still $n+\binom{k}{2}$. It would be very interesting to close this gap.

Clifton and Huang \cite{clifton-huang} conjectured that their upper bound $n+\binom{k}{2}$ for this hyperplane problem is tight if $n$ is sufficiently large with respect to $k$. Theorem \ref{thm-1} shows that this conjecture cannot be proved by following the approach of Clifton and Huang and only being more careful with the of the analysis of the coefficients of the polynomials appearing in their argument. In their approach, they consider a polynomial $f$ defined as the product of the linear hyperplane polynomials corresponding to a collection of hyperplanes satisfying the conditions. By applying the punctured higher-multiplicity version of the Combinatorial Nullstellensatz due to Ball and Serra \cite{ball-serra} to $f$, they obtain another polynomial $u$. This polynomial $u$ (and some of its derivatives) need to vanish at certain points, due to the higher-order vanishing properties of the polynomial $f$. Clifton and Huang show that these vanishing conditions for $u$ imply that $u$ must have sufficiently large degree (specifically, degree at least $3$ for $k=3$ and degree at least $5$ for $k=4$), and this gives their lower bound for the number of hyperplanes. In principle, one could hope to get better lower bounds from performing the analysis of the vanishing conditions for $u$ more carefully or for larger values of $k$. However, the second part of Theorem \ref{thm-1} implies that there exist polynomials of degree at most $2k-3$ satisfying the vanishing properties of the polynomial $u$ in the approach of Clifton and Huang. Hence one cannot prove a lower bound of $\binom{k}{2}$ on the degree of $u$ just by using the relevant vanishing conditions. Instead, when hoping to prove the conjecture that $n+\binom{k}{2}$ hyperplanes are necessary (if $n$ is sufficiently large with respect to $k$) with an argument along the lines of Clifton and Huang's approach, one would need to incorporate additional information about the polynomial $f$ from which $u$ is obtained. This polynomial $f$ is a product of linear polynomials, but this is not necessarily true for the polynomial $u$ obtained by applying the punctured higher-multiplicity version of the Combinatorial Nullstellensatz. Unfortunately, it is unclear how the fact that $f$ is a product of linear polynomials can be used when analyzing the polynomial $u$.

Note that Clifton and Huang's hyperplane problem is actually equivalent to the following problem: for fixed $k\geq 1$ and $n$ large with respect to $k$, what is the minimum possible degree of a polynomial $P\in \RR[x_1,\dots,x_n]$ with $P(0,\dots,0)\neq 0$ such that $P$ has zeroes of multiplicity at least $k$ at all points in $\{0,1\}^n\sm \{(0,\dots,0)\}$ and such that $P$ can be written as a product of linear polynomials? Without this last condition, this is precisely Problem \ref{problem-main}. By the results of Alon and Füredi \cite{alon-furedi}, the answers for both problems agree if $k=1$. The answers also agree for $k=2$ and $k=3$, since then $n+2k-3=n+\binom{k}{2}$. However, if Clifton and Huang's conjecture is true, then by Theorem \ref{thm-1} the answers must be different for $k\geq 4$. It would be interesting to prove (or disprove) that the two problems have different answers for sufficiently large $k$. 

One can also study a variant of Clifton and Huang's hyperplane problem, where one replaces the condition that no hyperplane contains $(0,\dots,0)$ by the condition that $(0,\dots,0)$ is covered by exactly $\ell$ hyperplanes for some given $0\leq \ell\leq k-1$. The case of $\ell=0$ corresponds to  Clifton and Huang's original hyperplane problem. As before, we can equivalently rephrase the problem in terms of polynomials, asking for the minimum possible degree of a polynomial $P\in \RR[x_1,\dots,x_n]$ with a zero of multiplicity exactly $\ell$ at $(0,\dots,0)$ and zeroes of multiplicity at least $k$ at all points in $\{0,1\}^n\sm \{(0,\dots,0)\}$ such that $P$ can be written as a product of linear polynomials. Again, one may ask whether the answer to this problem changes by omitting the last condition that $P$ is a product of linear polynomials. Theorems \ref{thm-2} and \ref{thm-3} determine the answer of the problem where the last condition is omitted (the answer is $n+2k-3$ for $0\leq \ell\leq k-2$ and $n+2k-2$ for $\ell=k-1$). By finding examples for the polynomial $P$ with the desired properties and of the appropriate degree such that $P$ is a product of linear polynomials, one can show that the answers for both problems agree for $k-3\leq \ell\leq k-1$. However, it is not clear what happens for smaller $\ell$.

Clifton and Huang also studied their hyperplane problem in the opposite parameter range, where the dimension $n$ is fixed and $k$ is large. They proved that for any fixed dimension $n$, the answer is of the form $(1+\frac{1}{2}+\dots+\frac{1}{n}+o(1))\cdot k$ as $k$ goes to infinity. It might also be interesting to study Problem \ref{problem-main} for fixed dimension $n$ and large $k$.

\subsection{Problem \ref{problem-main} over other fields}\label{subsec:other-fields}

One can also consider Problem \ref{problem-main} over other fields than $\RR$. For an arbitrary field $\FF$, let us say that a polynomial $P\in \FF[x_1,\dots,x_n]$ has a zero of multiplicity at least $k$ at a point $(a_1,\dots,a_n)\in \FF^n$ if the following holds: when expanding the polynomial $P(x_1+a_1,\dots,x_n+a_n)\in \FF[x_1,\dots,x_n]$, all monomials occurring in $P(x_1+a_1,\dots,x_n+a_n)$ have degree at least $k$. Note that for $\FF=\RR$ this agrees with our definition in terms of the derivatives of $P$ at $(a_1,\dots,a_n)$.

Our proof of Theorems \ref{thm-1} to \ref{thm-3} works for every field of characteristic $0$, but it is also interesting to consider Problem \ref{problem-main} over fields of positive characteristic. Since Alon and Füredi's result \cite{alon-furedi} stated in Theorem \ref{thm-alon-furedi} is valid over any field, one might also expect the answer for Problem \ref{problem-main} to be independent of the characteristic of the field.

Interestingly, this is not the case, and the answer for Problem \ref{problem-main} does depend on the characteristic of the field. Specifically, for a field $\FF$ of characteristic $p>3$ and $k=(p+5)/2$, there exists a polynomial $P\in \FF[x_1,\dots,x_n]$ of degree $\deg P\leq n+2k-4$ with $P(0,\dots,0)\neq 0$ and such that $P$ has zeroes of multiplicity at least $k$ at all points in $\{0,1\}^n\sm \{(0,\dots,0)\}$. In particular, the statement in Theorem \ref{thm-1} fails to hold in this case. Theorem \ref{thm-1} similarly fails for a field of characteristic $2$ and $k=4$ and for a field of characteristic $3$ and $k=7$.

For example, for $k=4$ and the field $\FF=\FF_2$, one can check that the following polynomial of degree $n+2k-4 = n+4$ has zeroes of multiplicity at least $4$ at all points in $\FF_2^n \sm \{(0,\ldots,0)\}$, but does not vanish at $(0,\ldots,0)$:
\[
    \left(\prod_{\ell=1}^n(x_\ell+1)\right) \cdot \left( 1+\sum_{i=1}^n (x_i^3 + x_i^2 + x_i)+\sum_{1 \leq i \neq j \leq n}(x_i^3+x_i^2)x_j + \sum_{1 \leq i<j\leq n} x_i x_j+\sum_{1 \leq i<j<k\leq n} x_i x_j x_k\right).
\]

The relevance of the values $k=(p+5)/2$ in characteristic $p>3$, as well as $k=4$ in characteristic $p=2$, and $k=7$ in characteristic $p=3$ is as follows: in each of these cases, $k$ is the smallest number such that the Catalan number $C_{k-2}$ is divisible by $p$. By using the arguments from our proof (together with some additional analysis of our map $\varphi_k$), one can show that there is a counterexample to Theorem \ref{thm-1} for this value of $k$ in each of these cases (but the theorem holds for all smaller values). 
In fact, the proof of Theorem \ref{thm-1} essentially gives an algorithm for producing such counterexamples.
It is worth pointing out that the only point in our proof which is dependent on the characteristic of the field $\FF$ is the assertion that the Catalan numbers are non-zero in $\FF$.

Since the first part of Theorem \ref{thm-1} fails to hold over fields of positive characteristic, the first part of Theorem \ref{thm-2} (which is a more general statement) also fails for over fields of positive characteristic. Theorem \ref{thm-3}, however, holds over any field.

\textit{Acknowledgments.} We thank the anonymous referees for their careful reading of this paper, and for many helpful comments. This work began while the first author was at Stanford University and was completed while the first author was at the Institute for Advanced Study.

\end{document}